\documentclass[reqno]{amsart}
\usepackage{amssymb,url}
\usepackage{amsmath,amsthm}
 \usepackage[abbrev]{amsrefs}

\DeclareMathOperator{\Eu}{Eu}
\DeclareMathOperator{\Vol}{Vol}
\DeclareMathOperator{\EVol}{EVol}
\DeclareMathOperator{\M}{M}
\DeclareMathOperator{\SM}{SM}
\DeclareMathOperator{\codim}{codim}
\DeclareMathOperator{\Grass}{Grass}
\DeclareMathOperator{\Conv}{Conv}
\DeclareMathOperator{\Osc}{Osc}
\DeclareMathOperator{\Ker}{Ker}
\DeclareMathOperator{\Ann}{Ann}
\DeclareMathOperator{\Coker}{Coker}

\theoremstyle{plain}
 \newtheorem{thm}{Theorem}
 \newtheorem{cor}[thm]{Corollary}
 
 \newtheorem{prp}[thm]{Proposition}
 
 \theoremstyle{definition}
 \newtheorem{eg}[thm]{Example}
  \newtheorem{rmk}[thm]{Remark}

\begin{document}

\author[R. Piene]{Ragni Piene}
\address{Department of Mathematics\\
University of Oslo\\ 
PO Box 1053 Blindern\\
NO-0316 Oslo, Norway}
\email{ragnip@math.uio.no}

\title{Higher order polar and reciprocal polar loci}

\begin{abstract}
In this note we introduce higher order polar  loci as natural generalizations of the classical polar loci, replacing the role of tangent spaces by that of higher order osculating spaces. The close connection between polar loci and dual varieties carries over to a connection between higher order polar loci and higher order dual varieties. We generalize the duality between the degrees of polar classes of a variety and those of its dual variety to varieties that are reflexive to a higher order. In particular, the degree of the  ``top'' (highest codimension) polar class of order $k$ is equal to the degree of the $k$th dual variety.
We define higher order Euclidean normal bundles and use them to define higher order reciprocal polar loci and classes. We give examples of how to compute the degrees of the higher order polar and reciprocal polar classes in some special cases: curves, scrolls, and toric varieties.
\vspace*{.5cm}

\end{abstract}

\maketitle

\centerline{\emph{To Bill Fulton on the occasion of his 80th birthday.}}
\bigskip

\section{Introduction}

The theory of polar varieties, or polar loci, has a long and rich history.  The terminology \emph{pole} and \emph{polar} goes back at least to Servois (1811) and Gergonne (1813). Poncelet's more systematic treatment of \emph{reciprocal polars} was presented in his  ``M\'emoire sur la th\'eorie g\'en\'erale des polaires r\'eciproques''  in 1824, though it did not appear in print until 1829. For a discussion of the -- at times heated -- debate between Poncelet and Gergonne concerning the principles of ``reciprocity'' versus that of ``duality'', see \cite{Bobillier}. Apparently, in the end, every geometer adopted both terms and used them interchangeably.
Less known, perhaps, is the work of Bobillier, who  (in 1828) was the first to replace the conic as  ``directrice'' by a curve of arbitrary degree $d$, so that the polar curve of a point is a curve of degree $d-1$, and the polar points of a line is the intersection of the polar curves of the points on the line, hence equal to $(d-1)^2$ points, the number of base points of the corresponding pencil. For a summary of this early history, see \cite{Bobillier} and \cite{MR846021}. 

In this note we introduce higher order polar  loci as natural generalizations of the classical polar loci. The close connection between polar loci and dual varieties carries over to a connection between higher order polar loci and higher order dual varieties. In particular, the degree of the  ``top'' (highest codimension) polar class of order $k$ is equal to the degree of the $k$th dual variety.

In a series of papers, Bank, Giusti, Heintz, Mbakop, and Pardo introduced what they called ``dual polar varieties'' and used them to find real solutions to polynomial equations. These loci were studied further in \cite{MR2397939} for real plane curves, and more generally in \cite{MR3335572}, under the name of ``reciprocal polar varieties.'' These variants of polar  loci are defined with respect to a quadric, in order to get a notion of orthogonality, and sometimes with respect also to a hyperplane at infinity. The orthogonality enables one to define \emph{Euclidean normal bundles}, as studied in \cite{MR1808617}, \cite{MR3451425}, and \cite{MR3335572}. For example, the (generic) Euclidean distance degree introduced in \cite{MR3451425} is the degree of  the ``top'' reciprocal polar class (see \cite{MR3335572}). Note that the definition of reciprocal polar loci given here differs slightly from the one given in 
 \cite{MR3335572}, but the degrees of the classes are the same.
 
 In the next section we recall the definition of the classical polar loci and their classes, and their relation to the Mather Chern classes. In the third section we define the higher order polar loci and their classes and discuss how the latter can be computed. In the case of a smooth, $k$-regular variety, the $k$th order polar classes can be expressed in terms of its Chern classes and the hyperplane class, and there are also other cases when it is possible to compute these classes. In the fourth section, we introduce the higher order Euclidean normal bundles and use them, in the following section, to define higher order reciprocal polar loci and classes. The three last sections give examples of how to compute the degrees of the higher order polar and reciprocal polar classes in some special cases: curves, scrolls, and toric varieties.
 
 \emph{Acknowledgments}. I am grateful to Patrick Popescu-Pampu for alerting me to the work of Bobillier and to the thesis \cite{Bobillier}. I thank the referee for asking a ``natural question,'' which is answered in Theorem \ref{degree}.

\section{Polar loci and Mather Chern classes}
Let $V$ be a vector space of dimension $n+1$ over an algebraically closed field of characteristic $0$.
The polar loci of an $m$-dimensional projective variety $X\subset \mathbb P(V)$ are defined as follows: Let $L_i\subset \mathbb P(V)$ be a linear subspace of codimension $m-i+2$. The \emph{$i$th polar locus} of $X$ with respect to $L_i$ is
 \[M_i(L_i):=\overline{\{P\in X_{\rm sm}\,|\, \dim (T_{X,P}\cap L_i)\ge i-1\}},\]
 where $T_{X,P}$ denotes the projective tangent space to $X$ at the (smooth) point $P$.  Note that $M_0(L_0)=X$.
 The rational equivalence class $[M_i(L_i)]$ is independent of $L_i$ for general $L_i$, and will be denoted $[M_i]$ and called the  \emph{$i$th polar class} of $X$ \cite[Prop.~(1.2), p.~253]{MR510551}. The classes $[M_i]$ are \emph{projective invariants} of $X$: the $i$th polar class of a (general) linear projection of $X$ is the projection of the $i$th polar class of $X$, and the $i$th polar class of a (general) linear section is the linear section of the $i$th polar class (see \cite[Thm. (4.1), p.~269; Thm. (4.2), p.~270]{MR510551}.
 
 Let us recall the definition of the Mather Chern class $c^{\M}(X)$ of an $m$-dimensional variety $X$. Let $\widetilde X\subseteq  \Grass_m (\Omega_X^1)$ denote the Nash transform of $X$, i.e., $\widetilde X$ is the closure of the image of the rational section $X\dashrightarrow \Grass_m (\Omega_X^1)$ given by the locally free rank $m$ sheaf $\Omega_X^1|_{X_{\rm sm}}$. 
The Mather Chern class of $X$ is $c^{\M}(X):=\nu_*(c(\Omega^\vee)\cap [\widetilde X])$, where $\Omega$ is the tautological sheaf on $\Grass_m (\Omega_X^1)$ and 
$\nu\colon \widetilde X\to X$. 

In 1978 we showed the following, generalizing the classical Todd--Eger formulas to the case of singular varieties:
 
 \begin{thm}\label{CP}\cite[Thm. 3]{MR1074588}
 The polar classes of $X$ are given by
  \begin{equation*}\label{CP1}
\textstyle [M_i]=\sum_{j=0}^i (-1)^{j}\binom{m-j+1}{m-i+1}h^{i-j}\cap c^M_{j}(X),
 \end{equation*}
 where $h:=c_1(\mathcal O_{\mathbb P^n}(1))$ is the class of a hyperplane.
 Vice versa, the Mather Chern classes of $X$ are given by
 \begin{equation*}\label{CP2}
 \textstyle c^{\M}_i(X)=\sum_{j=0}^i (-1)^{j}\binom{m-j+1}{m-i+1}h^{i-j}\cap [M_{j}].
 \end{equation*}
 \end{thm}

 \section{Higher order polar loci}\label{three}
 Let $X\subset \mathbb P(V)$ be a projective variety of dimension $m$,
and $P\in X$ a general point.  There is a sequence of osculating spaces to $X$ at $P$:
\[\{P\}\subseteq T_{X,P}=\Osc_{X,P}^1 \subseteq \Osc_{X,P}^2 \subseteq \Osc_{X,P}^3\subseteq \cdots \subseteq \mathbb P(V),\]
defined via the sheaves of principal parts of $\mathcal O_X(1)$ as follows. Let $V_X:=V\otimes \mathcal O_X$ denote the trivial $(n+1)$-bundle on $X$, and
consider the $k$-jet map (see e.g. \cite[p.~492]{MR0506323})
\[j_k\colon V_X\to \mathcal P^k_X(1).\]
Let $X_{k-{\rm cst}}\subseteq X$ denote the open  dense where the rank of $j_k$ is constant, say equal to $m_k+1$. Then for $P\in X_{k-{\rm cst}}$, $\Osc_{X,P}^k=\mathbb P((j_k)_P(V))\subseteq \mathbb P(V)$.
Note that $m_0=0$, $m_1=m$, and  $\dim \Osc_{X,P}^k=m_k$.

Assume $m_k<n$. Let $L_{k,i}\subset \mathbb P(V)$  be a linear subspace of codimension $m_k-i+2$. 
The \emph{$i$th polar locus of order} $k$ of $X$ with respect to $L_{k,i}$ is
 \[M_{k,i}(L_{k,i}):=\overline{\{P\in X_{k-{\rm cst}}\,|\, \dim (\Osc^k_{X,P}\cap L_{k,i})\ge i-1\}}.\]
 Note that $M_{1,i}(L_{1,i})=M_{i}(L_{1,i})$ is the classical $i$th polar locus with respect to $L_{1,i}$, and that $M_{k,0}(L_{k,0})=X$ for all $k$.

 The $(m_k+1)$-quotient $V_{X_{k-{\rm cst}}} \to j_k(V_{X_{k-{\rm cst}}})$ gives a rational section 
 \[X \dashrightarrow \Grass_{m_k+1}(V_X)=X\times \Grass_{m_k+1}(V).\]
  Its closure $\widetilde X^k \subseteq X\times\Grass_{m_k+1}(V)$, together with the projection map
$\nu^k\colon \widetilde X^k \to X$, is called the \emph{$k$th Nash transform} of $X\subset \mathbb P(V)$. 
Note that $\nu^1=\nu\colon \widetilde X^1=\widetilde X \to X$ is the usual Nash transform. Let $V_{\widetilde X^k} \to \mathcal P^k$ denote the induced $(m_k+1)$-quotient.  We call $\mathcal P^k$ \emph{the $k$th order osculating bundle} of $X$.
The projection map $\varphi_k\colon \widetilde X^k \to \Grass_{m_k+1}(V)$ is the \emph{$k$th associated map} of $X\subset \mathbb P(V)$ (see \cite{MR0154293} and \cite{MR2940698}).

\begin{thm}\label{p-class}
The class of $M_{k,i}(L_{k,i})$ is independent of $L_{k,i}$, for general $L_{k,i}$, and is given by
 \[[M_{k,i}] =\nu^k_*( c_i(\mathcal P^k)\cap [\widetilde X^k]).\] 
 \end{thm}
 
 \begin{proof} The proof is similar to that of \cite[Prop. (1.2), p.~253]{MR510551}.
 \end{proof}
 We call $[M_{k,i}] $ the \emph{$i$th polar class of order $k$} of $X$.
 
 \begin{prp}
 Let $V' \subseteq V$ be a general subspace, with $\dim V'> m_k+1$, and let $f\colon X\to \mathbb P(V')$ denote the corresponding linear projection. Then the image of the
 $i$th polar class of order $k$ of $X\subset \mathbb P(V)$ is the same as the $i$th polar class of order $k$ of $f(X)\subset \mathbb P(V')$.
 \end{prp}
 
 \begin{proof}
 The proof is similar to that of \cite[Thm. (4.1), p.~269]{MR510551}.
 \end{proof}
 
 In the case $k=1$, we proved \cite[Thm. (4.2), p.~270]{MR510551} the following result, which does not have an obvious generalization to the case $k\ge 2$.
 
 \begin{prp}
 Let $\mathbb P(W)\subset \mathbb P(V)$ be a (general) linear subspace of codimension $s$. Set $Y:=X\cap \mathbb P(W)$. For $0\le i\le m-s$, the $i$th polar class of $Y\subset \mathbb P(W)$ is equal to the intersection of the $i$th polar class of $X$ with $\mathbb P(W)$.
 \end{prp}
 
 Recall the definition of higher order dual varieties (introduced in \cite{MR713259}). 
 The points of the dual projective space $\mathbb P(V)^\vee:=\mathbb P(V^\vee)$ are the hyperplanes $H\subset \mathbb P(V)$.
 The \emph{$k$th order dual variety of $X$} is 
\[X^{(k)}:=\overline{\{ H\in \mathbb P(V)^\vee \,|\, H\supseteq \Osc^k_{X,P}, \text{ for some } P\in X_{k-{\rm cst}}}\}.\]
In particular, $X^{(1)}=X^\vee$ is the dual variety of $X$.

Set $ {\mathcal K}^k :=\Ker (V_{\widetilde X^k} \to \mathcal P^k)$; it is a $(n-m_k)$-bundle. Then
 $X^{(k)}\subset \mathbb P(V)^\vee$ is equal to the image of $\mathbb P(({\mathcal K}^k )^\vee)\subset \widetilde X^k\times \mathbb P(V)^\vee$ via the projection  on the second factor. Let
 $p: \mathbb P(({\mathcal K}^k )^\vee)\to X$ and $q: \mathbb P(({\mathcal K}^k )^\vee)\to X^{(k)}$  denote the projections.

\begin{prp}\label{pol}
 Let $L\subset \mathbb P(V)^\vee$  be a linear subspace of codimension $n-m_k+i-1$ and set $L_{k,i}:=L^\vee \subset \mathbb P(V)^{\vee \vee}=\mathbb P(V)$. Then
\[p(q^{-1}(X^{(k)}\cap L))=M_{k,i}(L_{k,i}).\]
\end{prp}

\begin{proof}
 If $L$ is general, then so is $L_{k,i}$. The dimension of $L_{k,i}$ is $n-1-\dim L=n-1-(m_k-i+1)=n-m_k+i-2$, hence
the codimension of $L_{k,i}$ is $m_k-i+2$. If $H\in X^{(k)}\cap L$, then there is a $P\in X_{k-{\rm cst}}$ such that $(P,H)\in \mathbb P(({\mathcal K}^k )^\vee)$. Hence $\Osc^k_{X,P}\subseteq H$ and 
$L_{k,i}\subseteq H$, so the intersection $\Osc^k_{X,P}\cap L_{k,i}$ has dimension $\ge i-1$. Therefore $P\in M_{k,i}(L_{k,i})$. Conversely, if $P\in M_{k,i}(L_{k,i})$, then $\dim \Osc^k_{X,P}\cap L_{k,i} \ge i-1$, so that  $\Osc^k_{X,P}$ and  $L_{k,i}$ do not span $\mathbb P(V)$. Hence there is a hyperplane $H$ that contains both these spaces, and so $H\in X^{(k)}\cap L$.
\end{proof}

The ``expected dimension'' of $X^{(k)}$ is equal to the dimension of $ \mathbb P(({\mathcal K}^k )^\vee)$, which is $m+n-m_k-1$. Let $\delta_k:=m+n-m_k-1-\dim X^{(k)}$ denote the \emph{$k$th dual defect} of $X$. Let $\overline{X}\subset \mathbb P(V)\times \mathbb P(V)^\vee$ denote the image of   $\mathbb P(({\mathcal K}^k )^\vee)$, and let $\overline{X^{(k)}}\subset \mathbb P(V)^\vee\times \mathbb P(V) \cong \mathbb P(V)\times \mathbb P(V)^\vee$ denote the corresponding variety constructed for $X^{(k)}$. 
It was shown in \cite[Prop.~1, p.~336]{MR713259}  that $\overline{X}\subseteq \overline{X^{(k)}}$, so that equality  holds iff their dimensions are equal. In this case, we say that $X$ is \emph{$k$-reflexive}, and we then
have $X=(X^{(k)})^{(k)}$. For example, a non-degenerate \emph{curve}  $X\subset \mathbb P(V)$ is $(n-1)$-reflexive
 (see Section \ref{curves}).

\begin{thm}\label{degree}
Assume $X$ is $k$-reflexive. Then the degree of the $i$th polar class of order $k$ of $X^{(k)}$ is given by
\[\deg [M^\vee_{k,i}]=\deg [M_{k,m-\delta_k-i}].\]
In particular, the degree of $X^{(k)}$ is equal to the degree of the $(m-\delta_k)$th polar class $[M_{k,m-\delta_k}]$ of order $k$ of $X$.
\end{thm}

\begin{proof}
For $k=1$, this is \cite[Thm.~(4), p.~189]{MR846021}, where it follows immediately from the definition of the ranks (corresponding to the degrees of the polar classes). Here we use Proposition \ref{pol}: if $h$ and $h^\vee$ denote the hyperplane classes of $\mathbb P(V)$ and $\mathbb P(V)^\vee$ respectively, then the class $[M_{k,m-\delta_k-i}]$ is the pushdown to $X$ of the class $(h^\vee)^{n-m_k+m-\delta_k-i-1}\cap [\overline X]$. By definition, $\delta_k=m+n-1-m_k-m^\vee$, where $m^\vee:=\dim X^{(k)}$, so this is the same as $(h^\vee)^{m^\vee -i}\cap [\overline X]$. Hence its degree is the degree of $h^{n-1-m_k^\vee+i}(h^\vee)^{m^\vee -i}\cap [\overline X]$, where $m_k^\vee$ denotes the dimension of a general $k$th osculating space of $X^{(k)}$. Similarly, the class $[M^\vee_{k,i}]$  is the pushdown to $X^{(k)}$ of the class $h^{n-m_k^\vee + i -1}\cap [\overline{X^{(k)}}]=h^{n-m_k^\vee + i-1}\cap [\overline{X}]$ and has degree
equal to the degree of $(h^\vee)^{m^\vee-i}h^{n-1-m_k^\vee+i}\cap [\overline X]$.

Note that $i=m-\delta_k$ is the largest $i\le m$ such that $[M_{k,i}]\ne 0$.
\end{proof}

\medskip

Assume $X\subset \mathbb P(V)$ is generically $k$-regular, i.e., the map $j_k\colon V_X\to \mathcal P^k_X(1)$ is generically surjective. The \emph{$k$th  Jacobian ideal} $\mathcal J_k$ is the $\binom{m+k}{k}$th Fitting ideal of $\mathcal P^k_X(1)$ \cite[(2.9)]{MR2940698}. Note that $\mathcal J_1=F^{m+1}(\mathcal P^1_X(1))=F^{m+1}(\mathcal P^1_X)=F^m(\Omega^1_X)$ is the ordinary Jacobian ideal. Let $\pi_k\colon X^k\to X$ denote the blow-up of $\mathcal J_k$. Then, by \cite[5.4.3]{MR0308104}, setting $\mathcal A_k:=\Ann_{\pi_k^*\mathcal P^k_X(1)}(F^{\binom{m+k}k}(\pi_k^*\mathcal P^k_X(1)))$, $\pi_k^*\mathcal P^k_X(1)/\mathcal A_k$ is a $\binom{m+k}k$-bundle.
Let $\mathcal I_k$ denote the $0$th Fitting ideal of the cokernel of the map $V_{X^k}\to \pi_k^*\mathcal P^k_X(1)/\mathcal A_k$ and $\overline \pi_k\colon \overline X^k\to X^k$ the blow-up of $\mathcal I_k$. Then  
 \cite[Lemma (1.1), p.252]{MR510551} the image $\overline{ \mathcal P}^k$ of $\overline V_{\overline X^k}$ in $\overline \pi_k\pi_k^*\mathcal P^k_X(1)/\mathcal A_k$ is a $\binom{m+k}k$-bundle. Hence we get a $\binom{m+k}k$-quotient $V_{\overline X^k} \to \overline{ \mathcal P}^k$ which agrees with $V_X\to \mathcal P^k_X(1)$ above $X_{k-{\rm cst}}$. 
 Note that, as discussed in the case $k=1$ in \cite[p.~255]{MR510551}, the map $\pi_k\circ \overline \pi_k\colon\overline X^k\to X$ factors via the $k$th Nash transform $\nu^k\colon \widetilde X^k\to X$. In particular, we have 
 \[ [M_{k,i}]=\pi_{k*}\overline \pi_{k*}(c_i(\overline{ \mathcal P}^k)\cap [\overline X^k]).\]

In some cases, the Chern classes of $\overline{ \mathcal P}^k$ can be computed in terms of the Chern classes of $\mathcal P^k_X(1)$ and the invertible sheaves $\mathcal J_k\mathcal O_{\overline X^k}$ and $\mathcal I_k\mathcal O_{\overline X^k}$. We shall see in Section \ref{curves} that this is the case when $X$ is a curve. Another case is the following.

\begin{prp}
Assume $X\subset \mathbb P(V)$ is smooth and generically $k$-regular, and that $m_k=n-1$. Then
\[[M_{k,i}]=c_1(\mathcal P^k_X(1))^i\cap [X] -\sum_{j=0}^{i-1} \binom{i}j c_1(\mathcal P^k_X(1))^j\cap s_{m-i+j}(I_k,X),\]
where $s_{m-i+j}(I_k,X)=-\overline \pi_{k*}(c_1(\mathcal I_k\mathcal O_{\overline X^k})^{i-j}\cap [\overline X^k])$ denote the Segre classes of the subscheme $I_k\subset X$ defined by the ideal $\mathcal I^k:= F^0(\Coker j^k)$.
\end{prp}

\begin{proof}
Since  $m_k=n-1$, the (locally free) kernel $\overline {\mathcal K}^k:=\Ker (V_{\overline X^k} \to \overline {\mathcal P}^k)$ has rank 1,  hence $c_i(\overline {\mathcal P}^k)=c_1(\overline {\mathcal P}^k)^i$ holds. We also have $\bigwedge^n \overline{\mathcal P}^k \cong \bigwedge^n \overline{\pi}_k^*\mathcal P^k_X(1) \otimes \mathcal I^k \mathcal O_{\overline X^k}$, hence $c_1(\overline{\mathcal P}^k)=c_1(\overline{\pi}_k^*\mathcal P^k_X(1)) +c_1(\mathcal I^k \mathcal O_{\overline X^k})$, so the result follows by applying the projection formula.
\end{proof}

\section{Higher order Euclidean normal bundles}

 Let $V$ and $V'$ be vector spaces of dimensions $n+1$ and $n$, and $V\to V'$ a surjection. 
  Set $H_\infty := \mathbb P(V')\subset \mathbb P(V)$, and call it the \emph{hyperplane at infinity}.
A non-degenerate quadratic form on $V'$ defines an isomorphism $V'\cong (V')^\vee$ and a non-singular quadric $Q_\infty\subset H_\infty$. 
\medskip

Let $L':=\mathbb P(W)\subset \mathbb P(V')$ be a linear space, and set 
\[K:=\Ker ({V'}^\vee \cong V' \to W)\subset {V'}^\vee.\]
 The \emph{polar} of $L'$ with respect to $Q_\infty$ is the linear space ${L'}^\perp:=\mathbb P(K^\vee) \subset \mathbb P(V')$.
 \medskip

Given a linear space $L\subset \mathbb P(V)$ and a point $P\in L$, define the  \emph{orthogonal space to $L$ at $P$} by
\[L_P^\perp := \langle P, (L\cap H_\infty)^\perp \rangle.\]
If $L\nsubseteq H_\infty$, the dimension of $(L\cap H_\infty)^\perp$ is equal to $n-1 -(\dim L -1)-1=n-\dim L -1$, and the dimension of 
$L_P^\perp$ is $n-\dim L $ if $P\notin (L\cap H_\infty)^\perp$ and $n-\dim L -1$ if $P\in (L\cap H_\infty)^\perp$. Note that if $P\in (L\cap H_\infty)^\perp$, then $P\in L\cap H_\infty$ and hence $P\in P^\perp$. This implies that $P\in Q_\infty$ and $L\cap H_\infty \subseteq T_{Q_\infty,P}$.

\medskip

Let $X\subset \mathbb P(V)$ be a variety of dimension $m$. Assume $X\nsubseteq H_\infty$. Let $P\in X$ be a non-singular point. The \emph{Euclidean normal space} to $X$ at $P$ with respect to $H_\infty$ and $Q_\infty$ is the linear space $N_{X,P}:= T_{X,P}^\perp$. Let $X_1\subseteq X$ denote the set of non-singular points $P$ such that $T_{X,P}\nsubseteq H_\infty$, and such that $T_{X,P}\cap H_\infty \nsubseteq T_{Q_\infty,P}$ if $P\in Q_\infty$. Then $\dim N_{X,P} = n-m$ for $P\in X_1$.
\medskip

By replacing tangent spaces by higher order osculating spaces we may define higher order Euclidian normal spaces. Let $k\ge 1$ be such that the dimension $m_k$ of a general $k$th order osculating space to $X$ is less than $n$. For $P\in X_{k-{\rm cst}}$, set 
\[N_{X,P}^k := (\Osc_{X,P}^k)^\perp.\]
Let $X_k\subseteq X_{k-{\rm cst}}$ denote the set of points $P$ such that $\Osc^k_{X,P}\nsubseteq H_\infty$, and such that $\Osc^k_{X,P}\cap H_\infty \nsubseteq T_{Q_\infty,P}$ if $P\in Q_\infty$. Then $\dim N^k_{X,P} = n-m_k$ for $P\in X_k$.

\medskip

With notations as in Section \ref{three}, consider the exact sequence
\[0\to \mathcal K^k \to V_{\widetilde X^k} \to \mathcal P^k \to 0.\]
Assume the hyperplane $H_\infty =\mathbb P(V')\subset \mathbb P(V)$ is \emph{general}.
  Then it follows from \cite[Lemma (4.1), p.~483]{MR0506323} that the dual  $V_{\widetilde X^k}'^\vee\to (\mathcal K^k)^\vee$ of the composed map $\mathcal K^k\to V_{\widetilde X^k} \to V'_{\widetilde X^k}$ is surjective.
  
\medskip

Set
  \[\mathcal X_k:=\overline{\{(P,P') | P\in X_k, P'\in N^k_{X,P} \}}\subset X\times \mathbb P(V).\]
 Let $p_k:\mathcal X_k\to X$ and $q_k:\mathcal X_k\to \mathbb P(V)$ denote the projections. The dimension of $\mathcal X_k$ is $n-m_k+m$. 
 The sheaf $p_{k*}q_k^*\mathcal O_{\mathbb P(V)}(1)$ is generically locally free, with rank $n-m_k+1$. Let $\nu_k:\widetilde X_k \to X$ be the ``Nash transform'' such that $\nu_k^*p_{k*}q_k^*\mathcal O_{\mathbb P(V)}(1)$ admits a locally free rank $n-m_k+1$ quotient bundle $\mathcal E^k$. By the definition of the normal spaces, it follows that $\mathcal E^k$ restricted to $X_k$ splits as $(\mathcal K^k)^\vee |_{X_k} \oplus \mathcal O_{X_k}(1)$, and that,
 by replacing  if necessary $\nu^k:\widetilde X^k \to X$ and $\nu_k: \widetilde X_k \to X$ by a further Nash transform (by abuse of notation, also denoted $\nu^k:\widetilde X^k\to X$), the map $V_{\widetilde X^k} \to  (\mathcal K^k)^\vee  \oplus \nu^{k*}\mathcal O_X(1)$ factors via the surjection $V_{\widetilde X^k} \to \mathcal E^k$. We call $\mathcal E^k$ the \emph{$k$th Euclidean normal bundle} to $X$. The difference between $\mathcal E^k$ and $(\mathcal K^k)^\vee  \oplus \nu^{k*}\mathcal O_X(1)$ measures where $V_{\widetilde X_k} \to  (\mathcal K^k)^\vee  \oplus \nu^{k*}\mathcal O_X(1)$ is not surjective.
 \medskip

  \begin{prp}\label{surjective}
 For a generic choice of a hyperplane $H_\infty$ and a quadric $Q_\infty\subset H_\infty$, the $k$th Euclidean normal bundle $\mathcal E^k$ splits as $(\mathcal K^k)^\vee \oplus \nu^{k*}\mathcal O_X(1)$  on an appropriate Nash modification $\nu^k:\widetilde X^k\to X$.
 \end{prp}

\begin{proof}
 We need to identify the points $P\in \widetilde X^k$ such that $\alpha:V_{\widetilde X^k}\to (\mathcal K^k)^\vee \oplus \nu^{k*}\mathcal O_X(1)$ is not surjective at $P$. First of all, by genericity of $H_\infty$, we may assume that $H_\infty$ intersects $X$ transversally, i.e., $H_\infty$ intersects each Whitney stratum of $X$ transversally. This means that $H_\infty$ does not contain a tangent space, nor a limit tangent space, to $X$ at any point of $X\cap H_\infty$. (This is the content of the surjectivity of $V'_{\widetilde X^1}\to (\mathcal K^1)^\vee$.) 
 This implies that $H_\infty$ does not contain a $k$th order osculating space, nor a limit of such a space, at ay point of $X\cap H_\infty$.  
 Hence $\alpha$ is surjective at all points of $\widetilde X^k$ above $X\setminus H_\infty$. Now consider $P\in \widetilde X^k$ mapping to $X\cap H_\infty$. The map $\alpha$ is not surjective at $P$ only if $\nu^k(P) \in (\widetilde O_P\cap H_\infty)^\perp$, where $\widetilde O_P$ is the limit $k$th osculating space to $X$ at $\nu^k(P)$ corresponding to $P$. But this implies that 
 $\nu^k(P)\in \nu^k(P)^\perp$, hence $\nu^k(P)\in Q_\infty$. Therefore $\nu^k(P)^\perp = T_{Q_\infty,\nu^k(P)}$, and hence
 $T_{Q_\infty,\nu^k(P)}\supseteq \widetilde O_P\cap H_\infty \supseteq \widetilde T_P \cap H_\infty$. But we may assume this does not happen: 
 by \cite[Cor.~4, p.~291]{MR360616}, a general quadric $Q_\infty\subset H_\infty$ intersects each of the Whitney strata of $X\cap H_\infty$ transversally.
 
\end{proof}

  The morphism $\varphi_k\colon \widetilde X^k\to \Grass_{m_k+1}(V)$ corresponding to the quotient $V_{\widetilde X^k}\to \mathcal P^k$ is the $k$th order associated map of $X\subset \mathbb P(V)$. The \emph{$k$th order associated normal map} is the morphism
  \[\psi_k\colon \widetilde X^k\to \Grass_{n-m_k+1}(V),\]
  defined by the quotient $V_{\widetilde X^k}\to \mathcal E^k$.
 
\section{Higher order reciprocal polar loci}

   Instead of imposing conditions on the osculating spaces of a variety, one can similarly impose conditions on the higher order Euclidean normal spaces. In what follows, we assume that $H_\infty$ and $Q_\infty$ are general with respect to $X$.
\medskip
   
  For each $i=0,\dots,m$ and $k\ge 1$, let $L^{k,i}\subset \mathbb P(V)$ be a general linear space of codimension $n-m_k+i$ and
  define the \emph{$i$th reciprocal polar locus of order $k$} with respect to $L^{k,i}$ to be
   \[M_{k,i}^\perp(L^{k,i}):=\overline{\{P\in X_{k} \, | \, N^k_{X,P}\cap L^{k,i} \neq \emptyset\}}.\]
 Note that, with $p_k:\mathcal X_k \to X$ and $q_k:\mathcal X_k \to \mathbb P(V)$ as in the previous section, $M_{k,i}^\perp(L^{k,i})=p_k q_k^{-1}(L^{k,i})$, and  $M_{k,0}^\perp(L^{k,0})=X$ for all $k\ge 1$.
  
 \begin{thm}
The class of $M_{k,i}^\perp(L_{k,i})$ is independent of $L_{k,i}$, for general $L_{k,i}$, and is given by
  \[[M_{k,i}^\perp]=\nu^k_*\bigl(\{c(\mathcal P^k)s(\nu^{k*}\mathcal O_X(1) )\}_i \cap [\widetilde X^k]\bigr).\]
 \end{thm}

 \begin{proof}
   Let $\Sigma_{k,i}(L^{k,i})\subset \Grass_{n-m_k+1}(V)$ denote the special Schubert variety consisting of the set of $(n-m_k)$-planes that meet the $(m_k-i)$-plane $L^{k,i}$. Then $M_{k,i}^\perp(L^{k,i})=\psi_k^{-1}(\Sigma_{k,i}(L^{k,i}))$. The first statement follows, with the same reasoning as in \cite[Prop. (1.2), p.~253]{MR510551}.

 Let $L^{k,i}=\mathbb P (W)$ and set $W'=\Ker (V\to W)$. The condition $N^k_{X,P}\cap L^{k,i} \neq \emptyset$ means that the rank of the composed map $W'_{\widetilde X^k} \to \mathcal E^k$ is $\le n-m_k$.
 By 
    Porteous' formula (see \cite[Thm.~14.4, p.~254]{MR732620}),
   $M_{k,i}^\perp(L^{k,i})$ has class
   \[[M_{k,i}^\perp]=\nu_*^k(s_i(\mathcal E^k)\cap [\widetilde X^k])=\nu^k_*\bigl(\{s((\mathcal K^k)^\vee)s(\nu^{k*}\mathcal O_X(1) )\}_i \cap [\widetilde X^k]\bigr).\]
Since the Segre class $s((\mathcal K^k)^\vee)$ is equal to the Chern class $c(\mathcal P^k)$,  the formula follows.
 \end{proof}  
 
 \begin{cor}
 Let  $h:=c_1(\mathcal O_X(1))$ denote the hyperplane class.
 We have
     \[[M_{k,i}^\perp]=\sum_{j=0}^i h^{i-j}\cap [M_{k,j}],\]
and hence
\[\deg [M_{k,i}^\perp]=\sum_{j=0}^i \deg [M_{k,j}].\]

\end{cor}
\begin{proof}
This follows, using Theorem \ref{p-class}, since
\[s(\mathcal O_{{\widetilde X}^k}(1) )=1+c_1(\mathcal O_{{\widetilde X}^k}(1) )+c_1(\mathcal O_{{\widetilde X}^k}(1) )^2+\cdots\]
and $\mathcal O_{{\widetilde X}^k}(1)=\nu^{k*}\mathcal O_X(1)$.
\end{proof}

\begin{rmk}
The (generic) \emph{Euclidean distance degree} of $X\subset \mathbb P(V)$, introduced in \cite{MR3451425}, can be interpreted as the degree of $[M_{1,m}^\perp]$. The above formula says that this degree is equal to the sum of the degrees of the polar classes of $X$, as stated in \cite[Thm.~5.4, p.~126]{MR3451425} (see also \cite{MR3335572}).
\end{rmk}

  \section{Curves}\label{curves}
 Let $X\subset \mathbb P(V)$ be a non-degenerate curve. At a general point $P\in X$ we have a complete flag
 \[\{P\}\subseteq T_{X,P}=\Osc_{X,P}^1 \subset \Osc_{X,P}^2 \subset  \cdots \subset \Osc_{X,P}^{n-1} \subset \mathbb P(V).\]
In this case $\nu^k: \widetilde X^k=\widetilde X \to X$ is the normalization of $X$,  $m_k=k$, and $\dim L_{k,1}=n-k-1$.

Note that the locus $M_{k,1}(L_{k,1})$ maps to $k$th order hyperosculating points on the image of $X$ under the linear projection $f\colon \mathbb P(V) \dashrightarrow \mathbb P(V')$ with center $L_{k,1}$, Namely, let $P\in M_{k,1}(L_{k,1})$. Then $\Osc_{X,P}\cap  L_{k,1}\ne \emptyset$, and hence $f(\Osc_{X,P})$  has dimension $<k$. This means that the $k$th jet map of $f(X)$ has rank $<k+1$ at the point $f(P)$. For example, for $k=1$, $f(P)$ is a cusp on $f(X)$, and for $k=2$, $f(P)$ is an inflection point.

Recall that the \emph{$k$th rank}  $r_k$ of the curve $X$ is defined as the degree of the $k$th osculating developable of $X$, i.e., as the number of $k$th order osculating spaces intersecting a given, general linear space of dimension $n-k-1$ \cite[pp.~199--200]{MR2850139}. Hence $r_k= \deg \nu^k_*(c_1(\mathcal P^k)\cap [\widetilde X])$, and  $r_k$ is also equal to the degree of the \emph{$k$th associated curve} of $X$ (see \cite[Prop.~(3.1), p.~480]{MR0506323} and \cite{MR0154293}, \cite{MR2940698}). 

Let $d:=r_0$ denote the degree of $X$.
Since $[M_{k,1}]=\nu^k_*(c_1(\mathcal P^k)\cap [\widetilde X])$, we get
$\deg [M_{k,1}]=r_k$. We have $r_k=(k+1)(d+k(g-1))-\sum_{j=0}^{k-1}(k-j)\kappa_j$, where $g$ is the genus of $\widetilde X$ and $\kappa_j$ is the $j$th stationary index of $X$ (see \cite[Thm.~(3.2), p.~481]{MR0506323}).
It follows that
 $\deg [M_{k,1}^\perp]=\deg [M_{k,0}]+\deg [M_{k,1}] =d+r_k$.
 For more on ranks, duality, projections, and sections, see \cite{MR0506323}. For example, the ranks of the \emph{strict dual curve} $X^{(n-1)}\subset \mathbb P(V)^\vee$ satisfy
 $r_k(X^{(n-1)})=r_{n-k-1}(X)$.

\begin{eg}
If $X\subset \mathbb P(V)$ is a \emph{rational normal curve} of degree $n$, then $\deg [M_{k,1}]=r_k=(k+1)(n-k)$ and
$\deg [M_{k,1}^\perp]=n+r_k=n+(k+1)(n-k)$.

Note that $X$ is \emph{toric} and \emph{$(n-1)$-self dual}: $X^{(n-1)}\subset \mathbb P(V)^\vee$ is a rational normal curve of degree $n$.
\end{eg}
 
 \section{Scrolls}
Ruled varieties -- scrolls -- are examples of varieties that are not generically $k$-regular, for $k\ge 2$. Hence we cannot hope to use the bundles of principal parts to compute the degrees of the higher polar classes for such varieties. However, in several cases we have results. In particular, since the degree of the 
``top'' $k$th order polar class is the same as the degree of the $k$th dual variety, we can consider the following situations.

\subsection{Rational normal scrolls}
Let $X=\mathbb P(\oplus_{i=1}^m \mathcal O_{\mathbb P^1}(d_i))\subset \mathbb P(V)$ be a rational normal scroll of type $(d_1,\dots, d_m)$, with $m\ge 2$,  $0<d_1\le \cdots \le d_m$, and  $n=\sum_{i=1}^m (d_i+1)-1$. Let $d:=d_1+\cdots +d_m$ denote the degree of $X$. The higher order dual varieties of rational normal scrolls were studied in \cite{MR738534}. For example, for $k$ such that $k\le d_1$, we have
\[\deg [M_{k,m}]=\deg X^{(k)}=kd-k(k-1)m,\]
where the first equality follows from Theorem \ref{degree} and the second was computed in \cite[Prop.~4.1, p.~389]{MR3694743}.
In the case $k=d_1=\cdots =d_m$, this gives 
\[\deg [M_{k,m}]=\deg X^{(k)}=kd-k(k-1)m=md_1^2-d_1(d_1-1)m=md_1=d.\]
Indeed, in this case $X$ is $k$-selfdual: $X^{(k)}\subset \mathbb P(V)^\vee$ is a rational normal scroll of the same type as $X$.

We refer to \cite{MR738534} and \cite{MR3694743} for other cases.

\subsection{Elliptic normal surface scrolls}
Higher order dual varieties of elliptic normal scrolls were studied in \cite{MR1143552}. 
Let $C$ be a smooth elliptic curve and $\mathcal E$ a rank two bundle on $C$. Assume $H^0(C,\mathcal E)\ne 0$, but that $H^0(C,\mathcal E\otimes \mathcal L)=0$ for all invertible sheaves $\mathcal L$ with $\deg \mathcal L<0$. Then let $e:=\deg \mathcal E$ denote the Atiyah invariant. There are two cases. Either $\mathcal E$ is decomposable -- then $\mathcal E=\mathcal O_C\oplus \mathcal L$ and $e=-\deg \mathcal L\ge 0$, or $\mathcal E$ is indecomposable, in which case $e=0$ or $e=-1$. Now let $\mathcal M$ be an invertible sheaf on $C$ of degree $d$. If $d\ge e+3$, then $\mathcal O_{\mathbb P(\mathcal E\otimes \mathcal M)}(1)$ yields an embedding of $\mathbb P(\mathcal E\otimes \mathcal M)$ as a linearly normal scroll $X\subset \mathbb P(V)\cong \mathbb P^{2d-e-1}$ of degree $2d-e$. The following holds \cite[Thm. 1, Thm. 2, p.~150]{MR1143552}.
\medskip

If $\mathcal E$ is decomposable:
\begin{itemize}
\item[(i)] if $e=0$, then $\deg [M_{d-1,2}]=\deg X^{(d-1)}=2d(d-1)$;
\item[(ii)] if $e=1$, then $\deg [M_{d-2,2}]=\deg X^{(d-2)}=2d^2-5d+2$;
\item[(iii)] if $e\ge 2$, then $\deg [M_{d-2,2}]=\deg X^{(d-2)}=d(d-1)$.
\end{itemize}

If $\mathcal E$ is indecomposable:
\begin{itemize}
\item[(i)] if $e=-1$, then $\deg [M_{d-1,2}]=\deg X^{(d-1)}=2d^2-3$;
\item[(ii)] if $e=0$, then $\deg [M_{d-1,2}]=\deg X^{(d-1)}=2d^2-d-2$.
\end{itemize}

\subsection{Scrolls over smooth curves}
Consider now the more general situation where $X\subset \mathbb P(V)$ is a smooth scroll of dimension $m$ and degree $d$ over a smooth curve $C$ of genus $g$. We showed in \cite{MR2369044} that in this case, for $k$ such that $m_k=km$, the $k$th jet map $j_k$ factors via a bundle $\mathcal P^k$ of rank $km+1$. Moreover, the Chern classes of  $\mathcal P^k$ can be expressed in terms of $d$, $m$, $k$, $g$, the class of a fiber of $X\to C$, and the class of a hyperplane section of $X$. Thus we can get a formula for the ``top'' $k$th order polar class in terms of these numbers and the Segre classes of the inflection loci of $X$, see \cite{MR2369044} for details.

\subsection{Scrolls over smooth varieties}
When we replace the curve $C$ by a higher-dimensional smooth projective variety of dimension $r$, the situation gets more complicated, but it is again possible to find a $k$th  osculating bundle $\mathcal P^k$ of rank $(m-r)\binom{r+k-1}r +\binom{r+k}r$, whose Chern classes can be computed, see \cite{MR3043593}.

 \section{Toric varieties}

The Schwartz--MacPherson Chern class of a    toric variety $X$ with torus orbits $\{X_\alpha\}_\alpha$ is equal to, by Ehlers' formula (see \cite[Lemma, p.~109]{MR1234037}, \cite[Thm., p.~188]{MR1197235}, \cite[Thm.~1.1; Cor.~1.2 (a)]{MR3417881}, \cite[Thm.~4.2, p.~410]{MR2209219}),
 \[\textstyle c^{\SM}(X)=\sum_\alpha [\overline X_\alpha].\] 
This implies (proof by the definition of $c^{\SM}(X)$ and induction on $\dim X$) that the Mather Chern class of a toric variety $X$ is equal to 
 \[\textstyle c^{\M}(X)=\sum_\alpha \Eu_X(X_\alpha)[\overline X_\alpha],\] 
  where  $\Eu_X(X_\alpha)$ denotes the value of the local Euler obstruction of $X$ at a point in the orbit $X_\alpha$.
 
 Therefore the polar classes of a toric variety X of dimension $m$ are
  \[\textstyle [M_i]=\sum_{j=0}^i (-1)^j \binom{m-j+1}{m-i+1} h^{i-j}\cap \sum_\alpha \Eu_X(X_\alpha)[\overline X_\alpha],\] 
  and the reciprocal polar classes
  \[ \textstyle [M_i^\perp]=\sum_{\ell=0}^i h^{i-\ell}\sum_{j=0}^\ell (-1)^j \binom{m-j+1}{m-\ell+1} h^{\ell-j}\cap \sum_\alpha \Eu_X(X_\alpha)[\overline X_\alpha]\] 
  \[\textstyle =\sum_{j=0}^i(-1)^j\sum_{\ell =0}^{i-j} \binom{m-j+1}{\ell}h^{i-j}\cap \sum_\alpha \Eu_X(X_\alpha)[\overline X_\alpha],\]
   where the second sum in each expression is over $\alpha$ such that $\codim X_\alpha=j$.

 It follows that if $X=X_\Pi$ is a projective toric variety corresponding to a convex lattice polytope $\Pi$, then (cf. \cite[Thm.~1.4, p.~2042]{MR2737807})
 \[\textstyle \deg [M_i] = \sum_{j=0}^i (-1)^j \binom{m-j+1}{m-i+1} \EVol^j(\Pi),\]
 and 
 \[\textstyle \deg [M_i^\perp] =\sum_{j=0}^i(-1)^j\sum_{\ell =0}^{i-j} \binom{m-j+1}{\ell}\EVol^j(\Pi),\]
 where ${\EVol}^j(\Pi):=\sum_\alpha \Eu_X(X_\alpha)\Vol(F_\alpha)$ denotes the sum of the  lattice volumes of the faces $F_\alpha$ of $\Pi$ of codimension $j$ weighted by the local Euler obstruction  $\Eu_X(X_\alpha)$ of $X$ at a point of $X_\alpha$, where $X_\alpha$ is the torus orbit of $X$ corresponding to the face $F_\alpha$. In particular, we get
 \[\textstyle \deg [M_m^\perp] =\sum_{j=0}^m(-1)^j (2^{m-j+1}-1)\EVol^j(\Pi),\]
  (cf. \cite[Thm.~1.1, p.~215]{MR3789441}).
 
 \begin{eg}
 N\o dland \cite[4.1]{MR3710713} studied weighted projective threefolds. In particular he showed the following. Assume $a,b,c$ are positive, pairwise relatively prime integers. The weighted projective threefold $\mathbb P(1,a,b,c)$ is the toric variety corresponding to the lattice polyhedron 
 $\Pi:=\Conv \{(0,0,0),(bc,0,0), (0,ac,0), (0,0,bc)\}$. It has isolated singularities at the three points corresponding to the three vertices other than $(0,0,0)$. Let $\Vol^j(\Pi)$ denote the sum of the lattice volumes of the faces of $\Pi$ of codimension $j$. We have 
 ${\Vol}^0(\Pi)=a^2b^2c^2$, ${\Vol}^1(\Pi)=abc(1+a+b+c)$,
 ${\Vol}^2(\Pi)=a+b+c+bc+ac+ab$, and ${\Vol}^3(\Pi)=4$. The algorithm given in \cite[A.2]{MR3710713} can be used to compute the local Euler obstruction at the singular points. N\o dland gave several examples of integers $a,b,c$ such that the local Euler obstruction at each singular point is $1$, thus providing counterexamples to a conjecture of Matsui and Takeuchi \cite[p.~2063]{MR2737807}. For example this holds for $a=2, b=3,c=5$, so in this case $\EVol^j(\Pi)=\Vol^j(\Pi)$ and we can compute
 \[\deg [M_0]=900, \deg [M_1]=3270, \deg [M_2]=4451,\deg [M_3]=2688,\]
 and hence
 \[ \deg [M_3^\perp]=11309.\]
\end{eg}
 \medskip
 
In general, formulas for the degrees of higher order polar classes and reciprocal polar classes of toric varieties are not known. However, in some cases, they can be found. As we have seen, for smooth varieties (not necessarily toric), if the $k$th jet map is surjective (i.e., the embedded variety is $k$-regular), then the classes $[M_{k,i}]$ can be expressed in terms of Chern classes of the sheaf of principal parts $\mathcal P^k_X(1)$. 
Hence, in the case of a $k$-regular toric variety, they can be expressed in terms of lattice volumes of the faces of the corresponding polytope. The following two examples were worked out in \cite[Rmk.~3.5, p.~385; Thm.~3.7, p.~387] {MR3330552}.

\begin{eg}
Let $\Pi\subset \mathbb R^2$ be a smooth lattice polygon with edge lengths $\ge k$. Then $X_\Pi$ is $k$-regular and 
\[\textstyle \deg [M_{k,1}]=\binom{k+2}2  \Vol^0 (\Pi) -\binom{k+2}3 \Vol^1 (\Pi) ,\]
\[\textstyle \deg [M_{k,2}]=\binom{k+3}4\bigl(3\Vol^0 (\Pi)-2k\Vol^1(\Pi)-\frac{1}3 (k^2-4)\Vol^2(\Pi)+4(k^2-1)\bigr).\]
\end{eg}

\begin{eg}
Let $\Pi\subset \mathbb R^3$ be a smooth lattice polyhedron with edge lengths $\ge 2$. Then $X_\Pi$ is 2-regular and 
\[ \deg [M_{2,1}]=4 \Vol^0 (\Pi) - \Vol^1(\Pi),\]
\[ \deg [M_{2,2}]=36 \Vol^0 (\Pi)- 27 \Vol^1 (\Pi)+6 \Vol^2 (\Pi) +18 \Vol^0(\Pi_0) +9 \Vol^1(\Pi_0) ,\]
\begin{multline*}
\deg [M_{2,3}]=62\Vol^0 (\Pi)-57\Vol^1(\Pi)+28\Vol^2(\Pi)-8\Vol^3(\Pi)\\+58\Vol^0(\Pi_0)+51\Vol^1(\Pi_0)+20\Vol^2(\Pi_0),
\end{multline*}
where $\Pi_0:=\Conv({\rm int}(\Pi)\cap \mathbb Z^3)$ is the convex hull of the interior lattice points of $\Pi$.
\end{eg}

Recall \cite[Def.~1.1, p.~1760]{MR3694743} that a variety $X\subset \mathbb P(V)$ is said to be \emph{$k$-selfdual} if there exists a linear isomorphism $\phi: \mathbb P(V)\xrightarrow{\sim} \mathbb P(V)^\vee$ such that $\phi(X)=X^{(k)}$. If $m=\dim X$, then $\deg [M_{k,m}]$ is the degree of the $k$-dual variety $X^{(k)}$. So if $X$ is $k$-selfdual, then $\deg [M_{k,m}]=\deg X$.  We refer to \cite{MR3694743} for examples of toric $k$-selfdual varieties.

\end{document}